\documentclass{amsart}
              [1997/02/02 v1.2e PROC Author Class]

\copyrightinfo{2006}
  {American Mathematical Society}

\newtheorem{theorem}{Theorem}[section]

\newtheorem{proposition}[theorem]{Proposition}

\begin{document}

\title[Separately continuous functions of $n$ variables with given restriction]{Construction of separately continuous functions of $n$
variables with given restriction}

\author{V.V.Mykhaylyuk}
\address{Department of Mathematics\\
Chernivtsi National University\\ str. Kotsjubyn'skogo 2,
Chernivtsi, 58012 Ukraine}
\email{vmykhaylyuk@ukr.net}

\subjclass[2000]{Primary 54C08, 54C30, 54C05}


\commby{Ronald A. Fintushel}


\keywords{separately continuous functions, first Baire class function, diagonal of mapping}

\begin{abstract}
It is solved the problem on construction of separately continuous functions on product of $n$ topological spaces with given restriction. In particular, it is shown that for every topological space $X$ and $n-1$ Baire class function $g:X\to \mathbb R$ there exists a separately continuous function $f:X^n\to\mathbb R$ such that $f(x,x,\dots,x)=g(x)$ for every $x\in X$.
\end{abstract}

\maketitle
\section{Introduction}

For every set $X$ and an integer $n\ge 2$ the mapping $d_n:X\to X^n$, $d_n(x)=(x,\dots,x)$, is called {\it diagonal mapping}. The set $\Delta_n=d_n(X)$ is called {\it the diagonal of the space $X$}, and the composition $g=f\circ d_n:X\to Y$ is called {\it the diagonal of mapping $f:X^n\to Y$}.

Let $X$ be a topological space. A mapping $f:X\to {\mathbb R}$ is called {\it a function of the first Baire class} if there exists a sequence $(f_n)_{n=1}^{\infty}$ of continuous function $f_n:X\to \mathbb R$ which converges to $f$ pointwise on $X$, i.e. $f(x)=\lim\limits_{n\to\infty}f_n(x)$ for every $x\in X$. Let $2\le\alpha <\omega_1$. A mapping $f:X\to\mathbb R$ is called {\it a function of $\alpha$-th Baire class} if there exists a sequence $(f_n)$ of at most Baire class $<\alpha$ functions $f_n:X\to\mathbb R$ which converges to $f$ pointwise on $X$.

R.~Baire shows in [1] that the diagonals of separately continuous functions of two real variables (i.e. functions which are continuous with respect each variable) are exactly the first Baire functions. A.~Lebesgue proved in [2] that every separately continuous function of $n$ real variables is a $(n-1)$-th Baire class function, in particular, its diagonal is a function of the same class. Conversely, it was shown in [3,4] that every real function of $(n-1)$-th Baire class is the diagonal of a separately continuous function of $n$ real variables.

Beginning from the second half of the 20-th century many mathematicians (see [5-10]) studied actively Baire classification of separately continuous functions and their analogs. Note that W.~Rudin firstly used a partition of the unit for the establishment of the belonging to the first Baire class of a separately continuous function defined on the product of metrizable and topological spaces and valued in a locally convex space. A development of the Rudin's method on the case of nonmetrizable spaces leads to the appearance of the following notion [11].

Topological space $X$ is called {\it a  $PP$-space} if there exist a sequence of locally finite covers $(U(n,i):i\in I_n)$ which of functionally open in $X$ sets $U(n,i)$  and sequence of families $(x(n,i):i\in I_n)$ of points $x(n,i)\in X$ such that for every $x\in X$ and neighborhood $U$ of $x$ there exists $n_0\in \mathbb N$ such that $x(n,i)\in U$ for every $n\ge n_0$ and $i\in I_n$ with $x\in U(n,i)$.

Class of $ÐÐ$-spaces is quite wide. It contain $\sigma$-metrizable paracompact spaces, topological vector spaces, which can be presented as the union of an increasing sequence of metrizable subspaces, Nemytski plane and Sorgenfrey line. It follows from [12] (see also [8, Theorem 3.14]) that for every $n\ge 2$ and
$PP$-space $X$ every separately continuous function $f:X^n\to\mathbb R$ is a function of the $(n-1)$-th Baire class and, in particular, its diagonal is a function of $(n-1)$-th Baire class.

The inverse problem on the construction of separately continuous function with the given diagonal was overlooked by mathematicians during a long time. This investigation was recrudescented by V.Maslyuchenko. The most general result for separately continuous functions of $n$ variables was obtained in [13]. It was obtained in [13] (see also [8, Theorem 3.24]) that for every function $g$ of $(n-1)$-th Baire class, which defined on a topological space $X$ with a normal $n$-th power and a $G_\delta$-diagonal $\Delta_n$ there exists a separately continuous function $f:X^n\to \mathbb R$ with the diagonal $g$.

On other hand, in the investigations of separately continuous functions $f:X\times Y\to \bf R$ defined on the product of topological spaces $X$ and $Y$ the following two topologies naturally arise (see [14]): the separately continuous topology $\sigma$ (the weakest topology with respect to which all functions $f$ are continuous) and the cross-topology $\gamma$ (it consists of all sets $G$ for which all $x$-sections $G^x=\{y\in Y:(x,y)\in G\}$ and $y$-sections
$G_y=\{x\in X:(x,y)\in G\}$ are open in $Y$ and $X$ respectively). The separately continuous topology $\sigma$ and the cross-topology $\gamma$ on the product $X_1\times X_2\times \cdots \times X_n$ of topological spaces $X_1, X_2, \dots , X_n$. Since the diagonal $\Delta=\{(x,x):x\in \bf R \}$ is a closed discrete set in $(\bf R^2,\sigma )$ or in $(\bf R^2,\gamma)$ and not every function defined on $\Delta$ can be extended to a separately continuous function on $\bf R^2$, even for $X=Y=\bf R$ the topologies $\sigma$ and $\gamma$ are not normal (moreover, $\gamma$ is not regular [14,15]). Thus, the construction separately continuous functions with the given diagonal is a partial case of more general problem: to establish for which subsets $E$ of a product $X_1\times X_2\times \cdots \times
X_n$ of topological spaces $X_1, X_2, \dots , X_n$ and $\sigma$-continuous or $\gamma$-continuous function $g:E\to \bf R$ there exists a separately continuous function $f:X\times Y \to \bf R$ for which the restriction $f|_E$ coincides with $g$.

This question for functions of two variables was study in [16]. It was obtained in [16] that for every topological space $X$ and a function $g:X\to\mathbb R$ of the first Baire class there exists a separately continuous function $f:X^2\to\mathbb R$ with the diagonal $g$.

In this paper analogously as in [16] we solve the problem on the construction of separately continuous function $f:X_1\times\cdots\times X_n\to\mathbb R$ with given restriction on a special type set $E\subseteq X_1\times\cdots\times X_n$. In particular, we obtain that the for a topological space $X$ and a function $g:X\to\mathbb R$ of $(n-1)$-th Baire class there exists a separately continuous function $f:X^n\to\mathbb R$ with the diagonal $g$.

\section{Notions and auxiliary statement}

{\it A set $A\subseteq X$ has the extension property in a topological space $X$}, if every continuous function $g:A\to [0,1]$ can be extended to a continuous function $f:X\to [0,1]$. According to Tietze-Uryson theorem [17, p.116], every closed set in a normal space has the extension property.

For a mapping $f:X\to Y$ and a set $A\subseteq X$ by $f|_A$ we denote the restriction of $f$ on $A$.

A set $A$ in a topological space $X$ is called {\it functionally closed} if there exists a continuous function $f:X\to [0,1]$ such that $A=f^{-1}(0)$.

A topological space $X$ is called {\it pseudocompact} if every continuous on $X$ function is bounded. 

A set $E$ in the product $X_1\times X_2\times\cdots\times X_n$ of topological spaces $X_1, X_2,\dots X_n$ is called {\it projectively homeomorphic} if for every $1\le i\le n$ the projection $p_i:E\to X_i$, $p_i(x_1,x_2, \dots ,x_n)=x_i$, is a homeomorphic embedding.

A set $E$ in the product $X_1\times X_2\times\cdots\times X_n$ is called {\it projectively injective} if $x_1\ne y_1$, $x_2\ne y_2$, ... , $x_n\ne y_n$ for every distinct points $(x_1,x_2,\dots,x_n), (y_1,y_2,\dots,y_n)\in E$, i.e. all projections of the set $E$ on axis $X_i$ are injective and {\it locally projectively injective} if for every $x\in X_1\times X_2\times\cdots\times X_n$ there exists a neighborhood $W$ of $x$ such that the set $E\cap W$ is projectively injective.

For a function $f:X\to \mathbb R$ by ${\rm supp} f$ we denote the set $\{x\in X: f(x)\ne 0\}$.

\begin{proposition}\label{p:1} Let $A$ be a functionally closed set which has the extension property in a topological space $X$, $1\le\alpha<\omega_1$ and $g:A\to\mathbb R$ be a function of $\alpha$-th Baire class. Then the function $f:X\to\mathbb R$ which defined by: $f(x)=g(x)$ for every $x\in A$ and $f(x)=0$ for every $x\in X\setminus A$, is a function of $\alpha$-th Baire class too.
\end{proposition}

\begin{proof} It follows from the definition of function of $\alpha$-th Baire class that it enough to prove the statement for $\alpha=1$.

Let $\alpha=1$ and $(g_n)_{n=1}^\infty$ be a sequence of continuous functions $g_n:A\to [-n,n]$, which pointwise converges to the function $g$. Since the set $A$ has the extension property in the topological space $X$, there exists a sequence $(f_n)_{n=1}^\infty$ of continuous function $f_n:X\to\mathbb R$ such that $f_n|_A=g_n$.

We take a continuous function $\varphi:X\to [0,1]$ such that $A=\varphi^{-1}(0)$ and for every $n\in\mathbb N$ and $x\in X$ we put $\varphi_n(x)=1-\min\{1,n\varphi(x)\}$. Clearly that all functions $\varphi_n$ are continuous on $X$, $A=\varphi_n^{-1}(1)$ and for every $x\in X\setminus A$ there exists an integer $m\in\mathbb N$ such that $\varphi_n(x)=0$ for every $n\ge m$. Then the sequence of continuous functions $f_n\cdot\varphi_n$ pointwise converges to the function $f$.
\end{proof}

\section{Main results}

\begin{theorem}\label{th:2} Let $E$ be a projectively homeomorphic set in the product $X_1\times \cdots \times X_n$ of topological spaces $X_1, \dots ,X_n$, moreover the projections $E_1,\dots,E_n$ of the set $E$ have the extension property in the spaces $X_1,\dots,X_n$ respectively and $g:E\to\mathbb R$ be a function of $(n-1)$-th Baire class. Then if $E$ is pseudocompact or all sets $E_1,\dots,E_n$ are functionally closed in $X_1,\dots,X_n$, theb there exists a separately continuous function $f:X_1\times\cdots\times X_n\to\mathbb R$ such that $f|_E=g$. 
\end{theorem}

\begin{proof} Firstly we consider the case of functionally closed sets $E_1,\dots,E_n$ in $X_1,\dots,X_n$ respectively.

Let $$f^{(1)}_{0}:X_1\to [0,1],\dots,f^{(n)}_0:X_n\to [0,1]$$ be continuous functions such that $$E_i=(f^{(i)}_0)^{-1}(0)$$ for every $1\le
i\le n$. For every $i=1,\dots,n$ and $x=(x_1,\dots,x_n)\in E$ we put $h_i(x_i)=x$. Clearly that $h_i$ is a homeomorphism of the set $E_i$ on the set $E$. Since $g$ is a function of $(n-1)$-th Baire class on $E$, there exists a family $(g_{k_1,\dots,k_{n-1}}:k_1,\dots,k_n\in{\mathbb N})$ of continuous functions $g_{k_1,\dots,k_{n-1}}:E\to\mathbb [-k_{n-1},k_{n-1}]$ such that 
$$
g(x)=\lim\limits_{k_1\to \infty} \lim\limits_{k_2\to \infty}\dots
\lim\limits_{k_{n-1}\to \infty}g_{k_1,k_2,\dots,k_{n-1}}(x)
$$
for every $x\in E$. For every $k_1,\dots,k_{n-1}\in\mathbb N$ and $1\le i\le n$ by $g_{k_1,\dots,k_{n-1}}^{(i)}$ we denote the continuous function $$g_{k_1,\dots,k_{n-1}}^{(i)}:E_i\to [-k_{n-1},k_{n-1}],\,\,\,g_{k_1,\dots,k_{n-1}}^{(i)}(x_i)=g_{k_1,\dots,k_{n-1}}(h_i(x_i))$$
and choose a continuous function $f_{k_1,\dots,k_{n-1}}^{(i)}:X_i\to\mathbb R$ such that $$f_{k_1,\dots,k_{n-1}}^{(i)}|_{E_i}=g_{k_1,\dots,k_{n-1}}^{(i)}.$$

We put $S=\{0\}\cup {\mathbb N}^{n-1}$. Further for every $s=(k_1,\dots,k_{n-1})\in S$ and $1\le i\le n$ the functions $f_{k_1,\dots,k_{n-1}}^{(i)}$, $g_{k_1,\dots,k_{n-1}}^{(i)}$ and $g_{k_1,\dots,k_{n-1}}$ we denote by $f_s^{(i)}$, $g_s^{(i)}$ and $g_s$.

We consider the continuous mappings $$\varphi_i=\mathop{\Delta}\limits_{s\in S}f_s^{(i)}:X_i\to
{\mathbb R}^S, \,\,\,\varphi_i(x_i)=(f_s^{(i)}(x_i))_{s\in S}.$$
We denote $Z=\bigcup\limits_{i=1}^n \varphi_i(X_i)$. Note that $Z$ is a metrizable space and for every $1\le i\le n$, $s\in {\mathbb N}^{n-1}$ and  $x=(x_1,\dots,x_n)\in E$ we have $f_s^{(i)}(x_i)=g_s^{(i)}(x_i)=g_s(x)$. Moreover for every $1\le i\le n$ the point $x_i$ from $X_i$ belongs to the set $E_i$
if and only if $\varphi_i(x_i)(0)=f_o^{(i)}(x_i)=0$. Therefore $\varphi_1(x_1)=\varphi_2(x_2)=\dots=\varphi_n(x_n)$ for every point $(x_1,x_2,\dots,x_n)\in E$, besides the set $A=\varphi_1(E_1)=\dots=\varphi_n(E_n)$ is functionally closed in $Z$.

We consider the function $\tilde{g}:A\to\mathbb R$, $\tilde{g}(z)=g(h_1(x_1))$, where $x_1\in E_1$ and $z=\varphi_1(x_1)$. Note that for every $x_1, y_1\in E_1$ the equality $\varphi_1(x_1)=\varphi_1(y_1)$ implies that $g^{(1)}_s(x_1)=g^{(1)}_s(y_1)$ for every $s\in S$. Then $g_s(h_1(x_1))=g_s(h_1(y_1))$ for every $s\in {\mathbb N}^{n-1}$ and $g(h_1(x_1))=g(h_1(y_1))$.  Thus, the definition of the function $\tilde{g}$ is correct.

For every $s=(k_1,\dots, k_{n-1})\in S$ the function $\tilde{f}_s:Z\to\mathbb R$, $\tilde{f}_s(z)=z(s)$, is continuous.
Then for every $z\in A$ we choose $x_1\in E_1$ such that $\varphi_1(x_1)=z$ and obtain
$$
\tilde{f}_s(z)=z(s)=f_s^{(1)}(x_1)=g_s^{(1)}(x_1)=g_s(h_1(x_1)).
$$
Therefore, 
$$
\lim\limits_{k_1\to\infty}\dots \lim\limits_{k_{n-1}\to\infty}\tilde{f}_s(z)= \lim\limits_{k_1\to\infty}\dots \lim\limits_{k_{n-1}\to\infty}
g_s(h_1(x_1))=g(h_1(x_1))=\tilde{g}(z).
$$
Thus, the function $\tilde{g}$ is a function of $(n-1)$-th Baire class on $A$. According to Proposition \ref{p:1}, if we put $\tilde{g}(z)=0$ for every $z\in Z\setminus A$, then we obtain the function $\tilde{g}$ of $(n-1)$-th Baire class on $Z$. It follows from [13, Theorem 2] that there exists a separately continuous function $\tilde{f}:Z^n\to\mathbb R$ with the diagonal $\tilde{g}$.

Consider the separately continuous function $$f:X_1\times\cdots \times X_n\to\mathbb R,\,\,\,f(x_1,\dots,x_n)=\tilde{f}(\varphi_1(x_1),\dots,\varphi_n(x_n)).$$
Let $x=(x_1,\dots,x_n)\in E$. Then $\varphi_1(x_1)=\dots=\varphi_n(x_n)\in A$ and
$$
f(x_1,\dots,x_n)=\tilde{f}(\varphi_1(x_1),\dots,\varphi_n(x_n))=\tilde{g}(\varphi_1(x_1))=g(h_1(x_1))=g(x).
$$
Thus, $g$ is the diagonal of the mapping $f$.

In the case of the pseudocompact set $E$ we can reasoning analogously. We consider the set $S={\mathbb N}^{n-1}$ and obtain continuous mappings $\varphi_i:E_i\to{\mathbb R}^S$. Then the set $A=\varphi_1(E_1)=\dots=\varphi_n(E_n)$ is a functionally closed set as a pseudocompact subsets in the metrizable space $Z$.
\end{proof}

In the case of $X_1=X_2=\dots=X_n=X$ ³ $E=\Delta_n$ we obtain the following result.

\begin{theorem} \label{th:3} Let $X$ be a topological space and $g:X\to\mathbb R$ be a function of $(n-1)$-th Baire class. Then there exists a separately continuous function $f:X^n\to\mathbb R$ with the diagonal $g$.\end{theorem}

\section{Functions on the product of compacts}

\begin{theorem} \label{th:4} Let $X_1$,...,$X_n$ be a compacts, $E$ be a closed projectively injective set in $X_1\times\cdots\times X_n$ and  $g: E\to\mathbb R$ be a function of $(n-1)$-th Baire class. Then there exists a separately continuous function $f:X_1\times\cdots\times X_n\to {\mathbb R}$ for which $f|_E=g$.
\end{theorem}

\begin{proof} Since the set $E$ is projectively injective, the projective mappings defined on the compact set $E$ and valued in $X_1, X_2,\dots , X_n$ are continuous injective mappings. Thus, they are homeomorphic embeddings. It remains to use Theorem \ref{th:2}.
\end{proof}

\begin{theorem} \label{th:5} Let $X_1$,...$X_n$ be a locally compact spaces such that $X=X_1\times\cdots\times X_n$ be a paracompact, $E\subseteq X$ be a closed locally projectively injective set and $g: E\to{\mathbb R}$ be a function of $(n-1)$-th Baire class. Then there exists a separately continuous function $f:X_1\times\cdots\times X_n\to {\mathbb R}$ with $f|_E=g$.
\end{theorem}

\begin{proof} For every point $p=(x_1,\dots,x_n)\in X_1\times\cdots\times X_n$ we choose open neighborhoods $U_p^{(1)},\dots , U_p^{(n)}$ of points $x_1,\dots, x_n$ in the spaces $X_1,\dots,X_n$ respectively such that the closures $$X_p^{(1)}=\overline{U}_p^{(1)}, \dots, X_p^{(n)}=\overline{U}_p^{(n)}$$ are compacts and the set $$E_p=E\cap (X_p^{(1)}\times \cdots \times X_p^{(n)})$$ is projectively injective. According to Theorem \ref{th:4}, for every $p\in X$ there exists a separately continuous function $f_p: X_p^{(1)}\times \cdots \times X_p^{(n)}\to{\mathbb R}$ with $f_p|_{E_p}=g|_{E_p}$. Since the space $X$ is a paracompact, there exists a partition of the unit $(\varphi_i: i\in I)$ on  $X$ which subordinated to the cover $(W_p=U_p^{(1)}\times \cdots\times U_p^{(n)}: p\in X)$ of the space $X$ [17, p.447]. For every $i \in I$ we choose $p_i\in X$ such that $\rm{supp}\varphi_i\subseteq W_{p_i}$ and put 
$$
g_i(x)=\left\{
\begin{array}{l}
f_{p_i}(x), \quad \mbox{if $x\in W_{p_i}$},\\
0, \quad \mbox{if $x\not\in W_{p_i}$}.
\end{array}
\right.
$$

Note that the function $\varphi_i\cdot g_i$ are separately continuous on $X$ and $(\varphi_i g_i)|_E=(\varphi_i |_E)g$. Then the function
$$\displaystyle f=\sum\limits_{i\in I}\varphi_i g_i$$ is a desired function.
\end{proof}

\bibliographystyle{amsplain}

\end{document}